\documentclass[]{colt2020}
\usepackage[utf8]{inputenc}

\DeclareMathOperator*{\argmin}{argmin}

\DeclareMathOperator*{\Tr}{Tr}
\DeclareMathOperator*{\expect}{\mathbb{E}}

\title{The Power of Linear Controllers in LQR Control}

\coltauthor{%
 \Name{Gautam Goel} \Email{ggoel@caltech.edu}\\
 \addr Caltech
 \AND
 \Name{Babak Hassibi} \Email{hassibi@caltech.edu}\\
 \addr Caltech%
}

\begin{document}
\date{}

\maketitle


\abstract{
 The Linear Quadratic Regulator (LQR) framework considers the problem of regulating a linear dynamical system perturbed by environmental noise. We compute the policy regret between three distinct control policies: i) the optimal online policy, whose linear structure is given by the Ricatti equations; ii) the optimal offline linear policy, which is the best linear state feedback policy given the noise sequence; and iii) the optimal offline policy, which selects the globally optimal control actions given the noise sequence. We fully characterize the optimal offline policy and show that it has a recursive form in terms of the optimal online policy and future disturbances. We also show that cost of the optimal offline  linear policy converges to the cost of the optimal online policy as the time horizon grows large, and consequently the  optimal offline linear policy incurs linear regret relative to the optimal offline policy, even in the optimistic setting where the noise is drawn i.i.d from a known distribution. Although we focus on the setting where the noise is stochastic, our results also imply new lower bounds on the policy regret achievable when the noise is chosen by an adaptive adversary. }

\section{Introduction}
\noindent In this paper we study control in linear dynamical systems. A system is initialized with state $x_0 \in \mathbb{R}^n$ and evolves according to the equation $$x_{t+1} = Ax_t + Bu_t + w_t,$$ where $A$ and $B$ are known $n \times n$  and $n \times m$ matrices and $w_t \in \mathbb{R}^n$ represents environmental noise. The variable $u_t \in \mathbb{R}^m$ represents a \textit{control action}; we can influence the evolution of the system by picking $u_t$ appropriately. At every step, we pay a state cost $c_t^x(x_t)$ as well as a control cost $ c_t^u(u_t)$, which are both usually assumed to be convex. The question we are interested in is how to pick the control actions so as to minimize our total cost over all rounds $t = 0 \ldots T -1$.

Control theorists have generally considered this problem in two distinct settings. In the $H_2$ (stochastic) setting, we assume that the noise $w = (w_0, \ldots w_{T-1})$ is a zero-mean noise variable with known distribution $\mathcal{D}$, and our goal is to minimize the expected aggregate cost across all rounds, 

\begin{eqnarray}
\min_{u_0, \ldots u_{T-1}} \hspace{1mm} \expect_{w\sim \mathcal{D}} \left[\sum_{t = 0}^{T-1} c_t^x(x_t) + c_t^u(u_t)  \right]. \label{stochastic-obj}
\end{eqnarray}

\noindent In the $H_{\infty}$ (adversarial) setting, the noise is assumed to be arbitrarily generated; the only assumption is that the noise is bounded, i.e. $\|w_t\|_2 \leq B$ for $t = 0 \ldots T - 1$. We seek a policy which minimizes the worst-case aggregate cost over bounded sequences of noise:  

\begin{eqnarray}
\min_{u_0, \ldots u_{T-1}} \hspace{1mm} \sup_{w} \left[\sum_{t = 0}^{T-1} c_t^x(x_t) + c_t^u(u_t) \right]. \label{adv-objective}
\end{eqnarray}

We can hence view $H_{\infty}$ control as a minimax game between the online controller and an adversarial environment whose goal is to make the controller incur as much cost as possible.

In this paper we adopt a different perspective from classical  control,  instead drawing from the online learning community. We consider control through the lens of \textit{regret minimization}. In regret minimization, the goal is design online control policies that approximate the performance that could have been achieved by the best controller (out of some class $\Pi$ of controllers), given access to the sequence of noise increments $w$ in advance. More precisely, we seek control policies that minimize the \textit{policy regret}: 

$$\min_{u_0, \ldots u_{T-1}} \hspace{1mm} \sup_{w} \left[\left(\sum_{t = 0}^{T-1} c_t^x(x_t) + c_t^u(u_t) \right) - \left(\sum_{t = 0}^{T-1} c_t^x(x_t^*) + c_t^u(u_t^*) \right)  \right].$$

\noindent Here $u^*$ is an \textit{optimal offline} sequence of control actions, and $x^*$ is the resulting sequence of states:

$$u_0^*, \ldots u_{T-1}^* = \argmin_{u_0, \ldots u_{T-1} \in \Pi} \sum_{t = 0}^{T-1} c_t^x(x_t) + c_t^u(u_t) \hspace{3mm} \text{where} \hspace{3mm}  x_{t+1} = Ax_t + Bu_t + w_t.$$

\noindent We emphasize that the the optimal offline sequence is defined with respect to both the class of policies $\Pi$ under consideration and the true sequence of realizations $w_0, \ldots w_{T-1}$; the optimal offline sequence is the cost-minimizing sequence of control actions given $w$, out of all sequences in the class $\Pi$.

A key advantage of the regret minimization perspective over classical control is that regret-minimizing controllers are \textit{adaptive}: they always achieve near-optimal performance relative to the best controller in the class $\Pi$, regardless of the how the noise is generated. This is in stark contrast to classical $H_2$ (resp. $H_{\infty}$) control theory, which produces controllers which perform well in the stochastic (resp. adversarial) regime, but whose performance can degrade badly if the noise is adversarial (resp. stochastic). The challenge in designing and analyzing online algorithms through the lens of regret is that regret is a \textit{counterfactual} performance metric: we compare the choices we made with limited information to the choices we could have made with full information, the latter set being potentially very different from the first. The control setting presents particular challenges when compared to classic  problems like Online Convex Optimization (OCO) and Multi-Armed Bandits (MAB), since the costs we incur in distinct rounds are coupled via the state; a poor decision in one round can steer the system into an undesirable trajectory, leading to heavy losses later on.

In this paper, we consider the problem of minimizing policy regret in the stochastic Linear Quadratic Regulator (LQR) setting, where the state costs and control costs are quadratic functions $c_t^x(x_t) = x^{\top}_tQx_t$ and $c_t^u(u_t) = u_t^{\top}Ru_t$ with $Q, R \succeq 0$, and the noise is picked i.i.d from a fixed distribution $\mathcal{D}$. We compare the performance of three distinct control policies:

\begin{enumerate}
    \item The \textbf{optimal online} policy. This is the policy which minimizes the expected aggregate cost (\ref{stochastic-obj}), out of all \textit{causal} policies, e.g. policies such that the control action $u_t$ depends only on the previously observed data $w_0 \ldots w_{t-1}, x_0 \ldots x_{t-1}$ and the current state $x_t$. This policy was originally derived in \cite{kalman1960contributions}, where it was shown that the optimal online policy has a linear structure: in every round, the cost-minimizing causal choice is to pick $u_t = -K_t x_t$ where the matrix $K_t$ can be found by solving the Ricatti equations, a system of linear recurrences in terms of the matrices $A, B, Q, R$.
    
    \item The \textbf{optimal offline linear} policy. This is the cost-minimizing linear state feedback policy $u_t = -K^* x_t$ where $$ K^* = \argmin_{K \in \mathbb{R}^{n \times m}}  \sum_{t = 0}^{T-1} x^{\top}_tQx_t + u^{\top}_tRu_t \hspace{3mm} \text{where} \hspace{3mm}  x_{t+1} = Ax_t + Bu_t + w_t, \hspace{3mm} u_t = -Kx_t. $$
    This policy is the optimal offline choice out of the class $\Pi_{\textsc{linear}}$, the class of linear state feedback controllers, e.g. controllers which always select a control action which is a fixed linear function of the state. We note that several recent papers focus on the problem of designing online learning algorithms which attain sublinear regret against this policy, e.g. \cite{agarwal2019online}, \cite{agarwal2019logarithmic}, \cite{cohen2018online}, \cite{abbasi2014tracking}. We also note that the problem of actually computing the optimal offline state feedback controller $K^*$ given the noise $w$ may be computationally intractable; we discuss this issue more thoroughly in Section \ref{offline-linear-sec}.
    
    \item The (unconstrained) \textbf{optimal offline} policy. This is the offline policy which selects the control actions $$u_0^*, \ldots u_{T-1}^* = \argmin_{u_0, \ldots u_{T-1} \in \mathbb{R}^m} \sum_{t = 0}^{T-1} x^{\top}_tQx_t + u^{\top}_tRu_t \hspace{3mm} \text{where} \hspace{3mm}  x_{t+1} = Ax_t + Bu_t + w_t.$$ Here the control actions are unconstrained; instead of being restricted to a class of policies $\Pi$, the control actions are selected as the global minimizers of the LQR objective, out of all possible control actions. This policy has also attracted recent attention, see for example \cite{goel2019online}, \cite{goel2019beyond}, \cite{li2019online}. While we might more properly refer to this policy as the unconstrained optimal offline policy, we will refer to this policy simply as the optimal offline policy for brevity.

\end{enumerate}

\subsection{Contributions of this paper}

We make three main contributions in this paper.

First, in Section \ref{offline-controller-sec} we derive the structure of the optimal offline policy, and show that it has an interesting recursive form in terms of the optimal online policy and the future noise (Theorem \ref{offline-structure-thm}). Our result parallels various results from the filtering literature, which express the solutions to smoothing problems (e.g. offline estimation) in terms of the corresponding filtering problems (e.g. online estimation) and future noise, see for example \cite{rauch1965maximum} and \cite{kailath2000linear} Sec. 10. We also compute the infinite horizon cost associated with the optimal offline policy (Theorem \ref{offline-cost-thm}). Our results close a gap left open by Kalman, who derived the optimal online policy and its infinite-horizon cost almost sixty years ago in \cite{kalman1960contributions}.

Second, in Section \ref{offline-linear-sec}, we compute the asymptotic cost of the optimal offline linear policy. Much recent work in the online learning community has focused on designing learning algorithms which can compete with this policy, albeit in the more challenging setting where the noise or cost functions is adversarial; we list several such works in Section \ref{related-work-sec}. We study this policy in the stochastic setting and compute its infinite-horizon cost. This result is highly nontrivial, since the offline optimal linear state feedback matrix $K^*$ is the minimizer of a polynomial whose degree scales with the time horizon $T$; since this optimization is highly non-convex, we have little hope of computing $K^*$ exactly. The polynomial is by necessity a random variable, since it depends on the noise realizations $w_0, \ldots w_{T-1}$. Our strategy is to show that in the asymptotic limit as $T$ tends to infinity the optimal offline linear cost converges almost surely to to the cost of the optimal online policy. To the best of our knowledge our proof technique is novel; we are not aware of any other work in the control or online learning community which computes an offline cost via a reduction to the online setting. We also prove a concentration inequality showing that the cost of the online optimal policy is tightly concentrated around its mean, a result which may be on independent interest to control theorists (Lemma \ref{concentration-lemma}).

Third, in Section \ref{policy-regret-sec} we apply our results to compute the pairwise policy regrets between all three policies. Our policy regret bound between the optimal online policy and the optimal offline policy is significant for two reasons. First, it is the first LQR policy regret bound we are aware of that compares an online policy to the (unconstrained) optimal offline policy, unlike much recent work which instead measures regret against the weaker optimal offline \textit{linear} policy. Second, while our bound is for the stochastic setting, it implies a lower bound on the best policy regret achievable in the adversarial setting; intuitively, giving an adversary control of the noise can only increase the regret incurred by the online learner. We also compute the policy regret between the optimal offline linear policy and the optimal offline policy, showing that it grows linearly in time. This suggests that the class of linear controllers is too restrictive to capture all of the performance offered by the offline optimal controller, and motivates performance metrics which are specifically designed to track the optimal offline cost, e.g. competitive ratio as considered in \cite{goel2019online}, \cite{goel2019beyond}, \cite{goel2017thinking}.

We emphasize two key strengths of our results. First, all of the theorems we prove hold in complete generality, and apply to any stabilizable linear dynamical system perturbed by i.i.d bounded noise; we impose no restrictive constraints on the underlying dynamical system or noise distribution. Second, all of the control costs we compute, as well as all of the policy regret bounds we derive, are exact: instead of merely bounding the costs and regrets of the various algorithms we consider, we give their exact numerical value.

\section{Related work} \label{related-work-sec}

\subsection{Optimal control}
 In the optimal control paradigm, we assume distributional knowledge of the noise $w$ and seek controllers which exactly minimize the expected LQR costs under this distribution; this is the setting we consider in this paper. We refer the reader to \cite{stengel1994optimal} for a survey of the vast optimal control literature.
We will often make use of Kalman's characterization of the optimal online LQR policy which he established in \cite{kalman1960contributions}:

\begin{theorem}(Kalman) \label{online-thm}

The online (i.e. strictly causal) policy which minimizes the infinite-horizon cost $$ \lim_{T \rightarrow \infty} \expect_{w \sim \mathcal{D}} \left[ \frac{1}{T} \sum_{t=0}^{T-1} x_t^{\top}Qx_t + u_t^{\top}Ru_t  \hspace{3mm} \text{where} \hspace{3mm}  x_{t+1} = (A-BK)x_t + w_t \right]$$ has the following linear structure: 
in every round, $u_t = -K x_t$, where $ K = (R + B^{\top}PB)^{-1}B^{\top}PA$ and $P$ is the unique p.s.d. solution of the algebraic Ricatti equation 

\begin{eqnarray} \label{are-def}
P = Q + A^{\top}PA - A^{\top}PB(R + B^{\top}PB)^{-1}B^{\top}PA.
\end{eqnarray}
Furthermore, the infinite-horizon cost under this policy is $\Tr(PW)$. 
\end{theorem}

\subsection{Online learning and control}
There has been much recent interest in control from the online learning community, much of it centered around  designing algorithms for LQR control with adversarial noise or costs that attains sublinear regret against the optimal offline linear policy, e.g. \cite{abbasi2011regret}, \cite{abbasi2014tracking},  \cite{cohen2018online}, \cite{agarwal2019online}, \cite{agarwal2019logarithmic}; these papers partially motivate our study of the optimal offline linear policy in the stochastic setting.
Many of these papers use classic techniques from the Online Convex Optimization (OCO) and bandits literature, such as Optimism in the Face of Uncertainty (OFU) and variations of Online Gradient Descent (see \cite{hazan2016introduction} for a survey). While these techniques are well-suited for the adversarial setting, we instead draw from the optimal control literature to understand the performance achievable in the stochastic setting.

\subsection{Competitive analysis}
A central focus of this paper is bounding the cost of an online control policy against the cost of the optimal offline  policy. In the online algorithms community, proving such bounds are the central aim of \textit{competitive analysis} (see \cite{borodin2005online} for a survey). We note that a series of recent papers also consider control-related problems from the perspective of competitive analysis, e.g. \cite{goel2017thinking}, \cite{goel2019online}, \cite{goel2019beyond}, \cite{li2018online}, \cite{li2019online}. Compared to our work, these papers usually give the online controller more power; for example, all of these papers assume that the online policy has predictions about the future noise. Furthermore, many of these papers assume that the control matrix $B$ is invertible, which is a very strong special case of controllablity; in this paper we only make the much weaker assumption that the system is stabilizable.

\section{Model and preliminaries}
\subsection{Control setting}
We formally define the control setting we study in this paper as follows. A linear system evolves according to the following dynamics equation: $$x_{t+1} = Ax_t + Bu_t + w_t, $$ where $x_t \in \mathbb{R}^n$ is the state variable, $u_t \in \mathbb{R}^m$ is a control variable, and $w_t \in \mathbb{R}^n$ is a noise variable. We assume without loss of generality that the initial point $x_0$ is zero. The matrices $A \in \mathbb{R}^{n \times n}$ and $B \in \mathbb{R}^{n \times m}$ are arbitrary, except that we assume the pair $(A, B)$ is \textit{stabilizable}, i.e. there exist matrices $K$ such that the $\rho(A - BK) < 1$; this condition is known as stability. A consequence of stability is that $A - BK$ is similar to a contraction matrix $L$, i.e.
$A - BK = MLM^{-1}$ where $ \|L\| \leq 1 - \gamma $ for some $\gamma \in (0, 1]$. We assume that the noise is stochastic and drawn i.i.d from a fixed distribution $\mathcal{D}$ with zero mean and bounded support, i.e. $\|w_t\|_2 \leq B$ for all $t \in 0 \ldots T-1$. 

We are interested in designing policies which minimize the expected LQR cost: 

\begin{eqnarray*}
\expect_{w \sim \mathcal{D}} \left[ \sum_{t=0}^{T-1} x_t^{\top}Qx_t + u_t^{\top}Ru_t \right],
\end{eqnarray*}

\noindent where $Q \succeq 0$ and $R \succ 0$. In this paper we are often in interested in the asymptotic behavior of the system, in the limit $T \rightarrow \infty$. In this setting the appropriate metric is the infinite-horizon LQR cost:

\begin{eqnarray*}
\lim_{T \rightarrow \infty} \expect_{w \sim \mathcal{D}} \left[ \frac{1}{T} \sum_{t=0}^{T-1} x_t^{\top}Qx_t + u_t^{\top}Ru_t \right].
\end{eqnarray*}

\noindent Notice that in this definition the cost is time-averaged, to prevent the cost from going to infinity.
We define the \textit{policy regret} between two control policies as the expected difference of their LQR costs; in the infinite-horizon setting we naturally define policy regret as the difference in their (time-averaged) infinite-horizon costs. If two policies have infinite-horizon policy regret converging to a constant $c_0 > 0$, then the finite-horizon policy regret between the two policies grows linearly at time at rate $c_0T$ (up to lower order terms).

We consider two distinct types of control policies: online policies (usually called strictly causal policies in the control literature), which in every round select a control action $u_t$ which depends on $x_0 \ldots x_t$ and $w_0 \ldots w_{t-1}$, and offline (non-casual) policies, which are free to pick actions which depend on the full sequence of states $x_0 \ldots x_{T-1}$ and the full sequence of noise $w = (w_0, \ldots w_{T-1})$. We note that in our online results we assume the controller picks the action $u_t$ after observing the state $x_t$ but before observing the noise $w_t$; this is more challenging than the setting considered in several recent papers, e.g. \cite{goel2019online}, \cite{goel2019beyond}, \cite{li2018online}, \cite{li2019online}, where the online policy observes $w_t$ before selecting $u_t$.

\subsection{Notation and terminology}
We often use $\|x\|^2_A$ as a shorthand for $x^{\top}Ax$. We let $\rho(A)$, $\kappa(A)$, and $\sigma_{\text{max}}(A)$ denote the spectral radius of a matrix $A$, its condition number, and its largest singular value, respectively.  We use the lowercase letters $x$ and $u$ to represent state and control variables, respectively, and reserve the capital letters $A, B, K, P, S$ to denote matrices associated with linear dynamical systems and their associated controllers; occasionally we use other capital letters to denote constants that appear in our bounds. We often refer to linear state feedback policies as linear policies. We use the terms ``control policy" and ``controller" interchangeably. In the special case where the control policy is a linear policy $u_t = -K x_t$, we may, via a slight abuse of terminology, refer to $K$ as the controller.


\section{The optimal offline policy} \label{offline-controller-sec}
In this section we derive the structure of the optimal offline controller, and show that it is intimately related to the structure of the optimal online controller derived by Kalman almost sixty years ago. Given a sequence $w = (w_0, \ldots w_{T-1})$, the optimal offline control actions $(u_0^*, \ldots u_{T-1}^*)$ are the ones which minimize the LQR objective

\begin{equation} \label{offline-prob}
     \left[ \left(\sum_{t=0}^{T-1} x_t^{\top}Qx_t + u_t^{\top}Ru_t\right) + x_T^{\top}Q_fx_T \hspace{3mm} \text{where} \hspace{3mm}  x_{t+1} =  Ax_t + Bu_t + w_t \right],
\end{equation}
where $Q_f$ represents a terminal state cost. We emphasize that the optimal offline control actions are defined with respect to the actual realizations $w_0, \ldots w_{T-1}$, instead of merely the noise distribution $\mathcal{D}$; the optimal offline control actions are the optimal actions in hindsight, with full knowledge of $w$.

\subsection{The structure of the optimal offline policy}

We use dynamic programming to recursively compute the optimal control actions, starting from the last time step and moving backwards in time; this approach mirrors Kalman's classic derivation of the optimal online policy in \cite{kalman1960contributions}.  For any fixed sequence of noise increments $w = (w_0, \ldots w_{T-1})$, define the ``offline cost-to-go" function $$V_t^w(x) = \min_u  [x^{\top}Qx + u^{\top}Ru + V_{t+1}^w(Ax + Bu + w_t)]$$ for $t = 1 .. \ldots T-1$, with $V_{T}(x) = x^{\top}Q_fx$. This function measures the aggregate cost over the future time horizon starting at the state $x$ at time $t$, under the assumption that in each time step, the offline controller picks the control action which minimizes the future cost given the current state and the realizations $w_t \ldots w_{T-1}$. 

We will show that that $V_t^w(x)$ can be written as $x^{\top}P_t x + v^{\top}_t x_t +  q_t$ for all $t \in [1\ldots T]$, where $P_t$ is defined as in the online policy. The claim is clearly true for $t = T$, since we can take $(P_{T}, v_T, q_T)  = (Q_f, 0, 0)$. Proceeding by backwards induction, suppose $V_{t+1}^w(x) = x^{\top}P_{t+1} x + v^{\top}_{t+1}x + q_{t+1}$ for some $ v_{T}, q_{T}$. We have 

$$
V_{t}^w(x) = \min_u  [x^{\top}Qx + u^{\top}Ru + (Ax + Bu + w_t)^{\top}P_{t+1}(Ax + Bu + w_t) + v_{t+1}^{\top}(Ax + Bu + w_t) + q_{t+1}].
$$
\noindent We can rewrite this more compactly in matrix form: 

\begin{equation*}
V_t^w(x) = \min_u 
\begin{pmatrix}
u^{\top} \\
x^{\top} \\
w_t^{\top} \\
v_{t+1}^{\top}
\end{pmatrix}
\begin{pmatrix}
R + B^{\top}P_{t+1}B & B^{\top}P_{t+1}A & B^{\top}P_{t+1} & \frac{1}{2}B^{\top} \\
A^{\top}P_{T}B & Q + A^{\top}P_{t+1}A & A^{\top}P_{t+1} & \frac{1}{2}A^{\top}\\
P_{T}B & P_{T}A & P_{T} & \frac{1}{2}I \\
\frac{1}{2}B & \frac{1}{2}A & \frac{1}{2}I & 0
\end{pmatrix}
\begin{pmatrix}
u \\
x \\
w_t\\
v_{t+1}
\end{pmatrix}
+ q_{t+1}
\end{equation*}

\noindent Using the Schur complement, we can make two observations. Firstly, the optimal offline control action in each round has the form 

\begin{eqnarray*}
u_t^* &=& -(R+B^{\top}PB)^{-1}B^{\top} \left(P_{t+1}Ax_t + P_{t+1}w_t + \frac{1}{2}v_{t+1} \right) \\
&=& - K_tx_t  -(R+B^{\top}P_{t+1}B)^{-1}B^{\top} \left(P_{t+1}w_t + \frac{1}{2}v_{t+1} \right),
\end{eqnarray*}
where $K_t$ is the optimal online controller originally computed by Kalman. In other words, the optimal offline control action at time $t$ is the sum of the optimal online control action and a term which depends only on the future disturbances $w_t \ldots w_{T-1}$.

\noindent Secondly, we can use the Schur complement to compute $V_t^w(x)$ explicitly:

\begin{eqnarray*}
V_t^w(x) &=& 
\begin{pmatrix}
x^{\top} \\
w_t^{\top} \\
v_{t+1}^{\top}
\end{pmatrix}
\begin{pmatrix}

 Q + A^{\top}P_{t+1}A & A^{\top}P_{t+1} & \frac{1}{2}A^{\top}\\
 P_{t+1}A & P_{t+1} & \frac{1}{2}I \\
 \frac{1}{2}A & \frac{1}{2}I & 0
\end{pmatrix}
\begin{pmatrix}
x \\
w_t\\
v_{t+1}
\end{pmatrix} \\
&& -
\begin{pmatrix}
x^{\top}A^{\top}P_{t+1} \\
w_t^{\top}P_{t+1} \\
\frac{1}{2} v_{t+1}^{\top}
\end{pmatrix}
B(R + B^{\top}P_{t+1}B)^{-1}B^{\top}
\begin{pmatrix}
x^{\top}A^{\top}P_{t+1} \\
w_t^{\top}P_{t+1} \\
\frac{1}{2}v_{t+1}^{\top}
\end{pmatrix}^{\top} + q_{t+1}
\end{eqnarray*}

\noindent Collecting terms, we see that $V_t^w(x) = x^{\top}P_t x + v_t^{\top}x + q_t$ where $P_t$ is the solution of the discrete time Ricatti equation obtained by Kalman, and $v_t$ and $q_t$ satisfy the recurrences


\begin{eqnarray}
v_t &=& 2A^{\top}S_t w_t + A^{\top}S_tP_{t+1}^{-1}v_{t+1} \label{v-recur} \\
q_t &=& w_t^{\top}S_{t+1}w_t + v_{t+1}^{\top}P_{t+1}^{-1}S_t w_t + q_{t+1} - \frac{1}{4}v_{t+1}^{\top} B (R + B^{\top} PB)^{-1} B^{\top} v_{t+1}  \label{q-recur},
\end{eqnarray}

\noindent  where we define 

\begin{eqnarray}
S_t = P_{t+1} - P_{t+1}B(R+ B^{\top}P_{t+1}B)^{-1}B^{\top}P_{t+1}. \label{s-def}
\end{eqnarray}

\noindent We have proven:

\begin{theorem} \label{offline-structure-thm}
Let $u_0^* \ldots u_{T-1}^*$ be the optimal offline control actions as defined in \ref{offline-prob}. These control actions have the following structure: for each $t \in [0, \ldots T-1]$, we have $$u_t^* = - K_tx_t  -(R+B^{\top}P_{t+1}B)^{-1}B^{\top} \left(P_{t+1}w_t + \frac{1}{2}v_{t+1} \right),$$ where $$K_t = (R+B^{\top}P_{t+1}B)^{-1}B^{\top}A,$$ $P_t$ is the solution of the discrete-time Ricatti recurrence, and $v_t$ satisfies the recurrence (\ref{v-recur}).
\end{theorem}

\noindent We note that this theorem parallels various results from the filtering literature, which express the solutions to smoothing problems (e.g. offline estimation) in terms of the corresponding filtering problems (e.g. online estimation) and future noise, see for example \cite{rauch1965maximum} and \cite{kailath2000linear} Sec. 10.

\subsection{The cost of the optimal offline policy}
Let us now turn to the problem of computing the infinite-horizon cost of the optimal offline policy we derived in Theorem \ref{offline-structure-thm}.  We prove:

\begin{theorem} \label{offline-cost-thm}
The infinite-horizon cost of the optimal offline policy described in Theorem \ref{offline-structure-thm} is
$$ \Tr(WS) -  \sum_{i=0}^{\infty} \Tr \left(W SA (A^{\top} - K^{\top}B^{\top})^i B (R + B^{\top} PB)^{-1} B^{\top} (A - BK)^i A^{\top}S \right)$$
where $W$ is the covariance of the noise, $P$ is the solution to the algebraic Ricatti Equation (\ref{are-def}), $S = P - PB(R + B^{\top}PB)^{-1}B^{\top}P$, and $K$ represents the optimal online policy in Theorem (\ref{online-thm}).
\end{theorem}

\begin{proof}
Using the notation we introduced in the proof of Theorem \ref{offline-structure-thm}, the infinite-horizon cost of the optimal offline policy is

$$\lim_{T \rightarrow \infty} \expect_{w \sim \mathcal{D}} \left[\frac{1}{T}V_0^w(x_0) \right] = \lim_{T \rightarrow \infty} \frac{1}{T} \expect_{w \sim \mathcal{D}} \left[x_0^{\top}P_0x_0 + v_0^{\top}x_0 + q_0 \right].$$ 

\noindent Recall that we assumed $x_0$. Using the recursion for $v_t$ given by (\ref{v-recur}) and the fact that $v_{T} = 0$ and $\expect[w_t] = 0$ for $t = 1 \ldots T-1$, we easily see that $\expect[v_t] = 0$ for all $t \in [0 \ldots T - 1]$. In particular, $\expect_w[v_0]=0$, so all that remains is to calculate $\expect[q_0]$. Using the recurrence (\ref{q-recur}) we derived for $q_t$, we see that

$$\expect[q_t] = \Tr(WS_t) - \frac{1}{4}\Tr(B (R + B^{\top} PB)^{-1} B^{\top} V_{t+1}) + \expect_w[q_{t+1}] $$

\noindent where we defined $V_t = \expect[v_t v_t^{\top}]$. Here we used the fact that $\expect[v_{t+1}^{\top}(P_{t+1}^{-1}S_t)w_t] = 0$, since $v_{t+1}$ and $w_t$ are independent and $\expect[w_t] = 0$. We have 

\begin{eqnarray*}
V_t &=&  \expect[v_t v_t^{\top}] \\
&=& \expect[(2A^{\top}S_t w_t + A^{\top}S_t P_{t+1}^{-1}v_{t+1})( 2A^{\top}S_t w_t + A^{\top}S_t P_{t+1}^{-1}v_{t+1})^{\top}] \\
&=& 4A^{\top}S_t W S_t A + A^{\top}S_t P_{t+1}^{-1}V_{t+1} P_{t+1}^{-1}S_t A,
\end{eqnarray*}

\noindent where we applied (\ref{v-recur}) and observed that the cross-terms vanish by independence of  $v_{t+1}$ and $w_t$ and the fact that $\expect[w_t] = 0$. 

Let us now consider the limiting behavior of $V_t$ as $t \rightarrow \infty$. It is well known that $P_t$ converges to $P$, the solution of the algebraic Ricatti equation (\ref{are-def}), as $t \rightarrow \infty$ (c.f. \cite{kailath2000linear}). Applying the definition of $S_t$ (\ref{s-def}), we see that $S_t$ converges to $$S = P - PB(R+ B^{\top}PB)^{-1}B^{\top}P.$$ To determine the convergence of $V_t$, it suffices to show that $\rho(A^{\top}S P^{-1}) < 1$ (see \cite{kailath2000linear}, Lemma D.1.2), in which case $V_t$ will converge to the solution of the equation 

\begin{eqnarray}
V = 4A^{\top}SWSA + A^{\top}SP^{-1}V P^{-1}SA \label{v-eq}.
\end{eqnarray}

\noindent Notice that $$ A^{\top}S P^{-1} = A^{\top} - A^{\top}PB(R + B^{\top}PB)^{-1}B^{\top} = (A-BK)^{\top},$$
where $K$ is represents the linear controller which minimizes the infinite-horizon cost. The Kalman gain $A - BK$ always has spectral radius strictly less than one, establishing the convergence of $V_t$ to the solution of (\ref{v-eq}), namely

$$V = 4\sum_{i=0}^{\infty} (A^{\top} - K^{\top}B^{\top})^i(A^{\top}SWSA)(A - BK)^i. $$

\noindent We see that the infinite-horizon optimal offline cost is 

\begin{eqnarray*}
\lim_{T \rightarrow \infty} \frac{1}{T}\expect_{w \sim \mathcal{D}}[V_0^w(x_0)] &=&  \lim_{T \rightarrow \infty} \frac{1}{T} \expect_{w \sim \mathcal{D}} [q_0]\\
&=& \lim_{T \rightarrow \infty} \frac{1}{T} \sum_{t=0}^{T-1} \left( \Tr(WS_t) - \frac{1}{4}\Tr(B (R + B^{\top} PB)^{-1} B^{\top}V_{t+1}) \right) \\
&=& \Tr(WS)  - \frac{1}{4}\Tr(B (R + B^{\top} PB)^{-1} B^{\top}V) \\
&=&  \Tr(WS) -  \sum_{i=0}^{\infty} \Tr \left(W SA (A^{\top} - K^{\top}B^{\top})^i B (R + B^{\top} PB)^{-1} B^{\top} (A - BK)^i A^{\top}S \right),
\end{eqnarray*}
where we plugged in the value of $V$ we obtained, and used the linearity and cyclic property of the trace.

\end{proof}

\section{The optimal offline linear policy} \label{offline-linear-sec}

In this section we compute the infinite-horizon cost of the optimal offline linear policy. Before we turn to this result, we note that is somewhat surprising that this cost can be computed at all. Recall that the evolution equation is $$x_{t+1} = Ax_t + Bu_t + w_t, $$ and suppose that the control policy is a linear state feedback policy, $u_t = -Kx_t$ for some $K \in \mathbb{R}^{n \times n}$. Iterating the dynamics backwards in time, we see that 
\begin{eqnarray}
x_{t} = \sum_{s=0}^{t-1} (A - BK)^{t - 1 - s} w_s. \label{state-eq}
\end{eqnarray}

\noindent Notice that $x_t$ depends on $K$ in a highly non-convex way; $x_t$ is a polynomial function of $K$ whose degree scales with $t$. It follows that the control variables and the LQR objective are also non-convex in $K$; in general, given the realizations $w_0, \ldots w_{T-1}$, it is not clear how to compute the offline optimal linear policy $K^*$, since this involves minimizing a polynomial of degree $T-1$. Nevertheless, we compute the infinite-horizon cost of this policy. Our strategy is to show that as $T$ grows large, the cost of optimal offline linear converges to the cost of the optimal online policy. Intuitively, each realization $w_t$ makes little difference in the asymptotic limit, so the offline cost converges to its expectation, which is the cost of the online policy. Superficially, our result resembles the Law of Large Numbers, but we emphasize a key difference: in LLN-type results the summands are usually i.i.d, but in the control setting the costs may be highly correlated across time, since the costs all depend on the state.
 We prove:

\begin{theorem} \label{offline-linear-cost-thm}
Let $(A, B)$ be any stabilizable pair of matrices. Consider the linear dynamical system given by $$x_{t+1} = Ax_t + Bu_t + w_t, $$ where the noise $w$ is drawn i.i.d from a fixed distribution $\mathcal{D}$ with zero mean and bounded support. In this dynamical system, the cost of the optimal offline linear policy converges almost surely to the cost of the optimal online policy as $T \rightarrow \infty$:

$$  \min_{K \in \mathbb{R}^{n \times m}} \frac{1}{T} \sum_{t=0}^{T-1} x_t^{\top}Qx_t + u_t^{\top}Ru_t    \hspace{3mm}  \overset{a.s.}{\rightarrow} \hspace{3mm} \Tr(PW).$$

\end{theorem}


\begin{proof}

\noindent Let $K \in \mathbb{R}^{m \times n}$ be any matrix so that $A - BK$ is stable $(\rho(A - BK) < 1)$. Recall that this implies that there exists a matrix $L$ and a similarity transform $M$ so that $A - BK = MLM^{-1}$ and $\|L\| = 1 - \gamma$ where $\gamma \in (0, 1]$. 

Define the function $$\textsc{cost}_T(K; w_0 \ldots w_{T-1}) = \frac{1}{T} \sum_{t=0}^{T-1} x_t^{\top}(Q + K^{\top}RK)x_t  \hspace{3mm} \text{where} \hspace{3mm}  x_{t+1} = (A-BK)x_t + w_t.$$

\noindent This function measures the time-averaged LQR cost of the linear policy $u_t = -Kx_t$ on the instance $w_0 \ldots w_{T-1}$. Similarly, define the function
$$ \textsc{\textsc{cost}}(K) = \expect_w \left[ \frac{1}{T} \sum_{t=0}^{T-1} x_t^{\top}(Q + K^{\top}RK)x_t  \hspace{3mm} \text{where} \hspace{3mm}  x_{t+1} = (A-BK)x_t + w_t \right]. $$

\noindent This function measures the expected infinite-horizon cost of the linear policy represented by $K$. The key difference between $\textsc{cost}_T(K; w_0 \ldots w_{T-1})$ and $\textsc{cost}$ is that the former is the cost of the policy $K$ on a specific instance $w = (w_0, \ldots, w_{T-1})$, whereas the latter cost is not defined relative to any specific instance but is rather the expected cost of the policy $K$, averaged over all instances $w$.

Using equation (\ref{state-eq}), we can rewrite $\textsc{cost}_T$ as  
$$\textsc{cost}_T(K; w_0 \ldots w_{T-1}) = \frac{1}{T}\sum_{t=0}^{T-1} \left \| \sum_{s=0}^{t-1} (A - BK)^{t -1 - s} w_s \right\|^2_{Q + K^{\top}RK}.$$

\noindent We first show that $\textsc{cost}_T$ is a \textit{bounded differences function} when restricted to the set of $w$ such that $\|w_t\| \leq B$ for all $t \in [0 \ldots T-1]$. Formally, that means that the following: for every $i \in 0 \ldots T-1$ and all fixed $w_0, \ldots w_{i-1}, w_{i + 1}, \ldots w_{T-1}$, there exists some $c_i$ such that $$\Delta_i := \sup_{w_i, w_i'} \left[\textsc{cost}_T(K; w_0, \ldots  w_i, \ldots w_{T-1}) -   \textsc{cost}_T(K; w_0, \ldots  w_i', \ldots w_{T-1}) \right] \leq c_i.$$ 

\noindent Intuitively, this means that changing $w$ in any single coordinate cannot change the value of $f_T(K; \cdot)$ too much.  We bound $\Delta_i$ as follows:

\begin{eqnarray*}
\Delta_i &=&
\frac{1}{T}\sum_{t=i+1}^{T-1} \left \|(A - BK)^{t -1 - i} w_i + \sum_{s=0, s\neq i}^{t-1} (A - BK)^{t -1 - s} w_s \right \|^2_{Q + K^{\top}RK} \\
&& -\frac{1}{T}\sum_{t=i+1}^{T-1} \left \|(A - BK)^{t -1 - i} w_i' + \sum_{s=0, s\neq i}^{t-1} (A - BK)^{t -1 - s} w_s \right \|^2_{Q + K^{\top}RK} \\
&=& \frac{1}{T}\sum_{t=i+1}^{T-1} \left \|(A - BK)^{t -1 - i} w_i \right\|^2_{Q + K^{\top}RK} - \left \|(A - BK)^{t -1 - i} w_i' \right\|^2_{Q + K^{\top}RK} \\
&& + \frac{2}{T}\sum_{t=i+1}^{T-1} \left((A - BK)^{t -1 - i} (w_i - w_i') \right)^{\top}(Q + K^{\top}RK) \left( \sum_{s=0, s\neq i}^{t-1} (A - BK)^{t -1 - s} w_s \right)\\
&\leq& \frac{1}{T}\sum_{t=i+1}^{T-1}  \|QL^{t - 1 - i}Q^{-1}\|^2 \|w_i\|^2 \|Q + K^{\top}RK\| \\
&& + \frac{2}{T}\sum_{t=i+1}^{T-1} \|QL^{t-1-i}Q^{-1} \| \|w_i - w_i'\| \|Q + K^{\top}RK\| \left( \sum_{s=0, s\neq i}^{t-1} \|QL^{t-1-s}Q^{-1} \| \|w_s\| \right)\\
&\leq& 5B^2 \kappa^2(M) \sigma_{\text{max}}(Q + K^{\top}RK) \frac{1}{\gamma^2T},  \\
\end{eqnarray*}

\noindent where we used the boundedness of $w$, stability of $A - BK$, and the formula for the sum of a geometric series. Since the $w_i$ are assumed to be independent, we can immediately apply McDiarmid's Inequality ( \cite{mcdiarmid1989method}) to obtain:

\begin{lemma} \label{concentration-lemma}
For all $K \in \mathbb{R}^{n \times m} $ such that $A - BK$ is stable, the function $\textsc{cost}_T(K; w)$ obeys the following concentration inequality:

$$\Pr \left( \left|\textsc{cost}_T(K; w) - \expect_w[\textsc{cost}_T(K;w)] \right | \geq \epsilon \right) \leq 2\exp{\left(-\frac{2\epsilon^2 \gamma^4T }{25B^4 \kappa^4(M) \sigma^2_{\text{max}}(Q + K^{\top}RK)} \right)}.$$
\end{lemma}

\noindent Note that as $T$ tends to infinity, $\textsc{cost}_T$ becomes more and more sharply concentrated around its mean. This implies that for all stabilizing $K$, the r.v. $\textsc{cost}_T(K; w_0 \ldots w_{T-1})$ converges pointwise to the expected infinite-horizon cost under the linear policy represented by $K$:

$$\textsc{cost}_T(K; w_0 \ldots w_{T-1})  \hspace{3mm} \overset{a.s.}{\rightarrow}  \hspace{3mm}\textsc{cost}(K).$$

\noindent Since $\textsc{cost}_T$ and $\textsc{cost}$ are both smooth functions of $K$, this implies that 

$$\min_K \textsc{cost}_T(K; w_0 \ldots w_{T-1})  \hspace{3mm} \overset{a.s.}{\rightarrow} \hspace{3mm}\min_{K}   \textsc{cost}(K).$$

\end{proof}

\section{Policy regret bounds} \label{policy-regret-sec}

The computation of the pairwise policy regrets between the three policies we consider follows immediately from Theorems \ref{online-thm}, \ref{offline-cost-thm}, and \ref{offline-linear-cost-thm}:

\begin{theorem}
As $T \rightarrow \infty$, the pairwise policy regrets between the optimal online policy, the optimal offline linear policy, and the optimal offline policy exhibit the following behavior: 
\begin{enumerate}
    \item The time-averaged policy regret between the optimal online policy and the offline optimal policy and the time-averaged policy regret between the optimal offline linear policy and the offline optimal policy both converge to $$ \Tr(W(P-S)) +  \sum_{i=0}^{\infty} \Tr \left(W  SA(A^{\top} - K^{\top}B^{\top})^i B (R + B^{\top} PB)^{-1} B^{\top} (A - BK)^i  A^{\top}S \right). $$
    
    \item The time-averaged policy regret between the optimal online policy and the offline optimal linear policy converges to zero.
    
\end{enumerate}

\end{theorem}

\noindent We note that the first part of this theorem also gives a lower bound on the policy regret between the optimal online policy and the optimal offline policy in the setting where the noise is adversarial, since clearly

\begin{eqnarray*}
\min_{u \in \Pi} \expect_{w \sim \mathcal{D}} \left[ \frac{1}{T} \sum_{t=0}^{T-1} x_t^{\top}Qx_t + u_t^{\top}Ru_t  - \min_u  \frac{1}{T} \sum_{t=0}^{T-1}  x_t^{\top}Qx_t + u_t^{\top}Ru_t \right] \\
\leq \hspace{2mm} \min_{u \in \Pi} \sup_{w \in \Lambda} \left[ \frac{1}{T} \sum_{t=0}^{T-1} x_t^{\top}Qx_t + u_t^{\top}Ru_t  - \min_u  \frac{1}{T} \sum_{t=0}^{T-1} x_t^{\top}Qx_t + u_t^{\top}Ru_t \right],
\end{eqnarray*}
where  $\Pi$ is the class of causal policies and  $\Lambda$ is any class of bounded disturbances.

\newpage 

\bibliography{references.bib}

\begin{thebibliography}{17}
\providecommand{\natexlab}[1]{#1}
\providecommand{\url}[1]{\texttt{#1}}
\expandafter\ifx\csname urlstyle\endcsname\relax
  \providecommand{\doi}[1]{doi: #1}\else
  \providecommand{\doi}{doi: \begingroup \urlstyle{rm}\Url}\fi

\bibitem[Abbasi-Yadkori and Szepesv{\'a}ri(2011)]{abbasi2011regret}
Yasin Abbasi-Yadkori and Csaba Szepesv{\'a}ri.
\newblock Regret bounds for the adaptive control of linear quadratic systems.
\newblock In \emph{Proceedings of the 24th Annual Conference on Learning
  Theory}, pages 1--26, 2011.

\bibitem[Abbasi-Yadkori et~al.(2014)Abbasi-Yadkori, Bartlett, and
  Kanade]{abbasi2014tracking}
Yasin Abbasi-Yadkori, Peter Bartlett, and Varun Kanade.
\newblock Tracking adversarial targets.
\newblock In \emph{International Conference on Machine Learning}, pages
  369--377, 2014.

\bibitem[Agarwal et~al.(2019{\natexlab{a}})Agarwal, Bullins, Hazan, Kakade, and
  Singh]{agarwal2019online}
Naman Agarwal, Brian Bullins, Elad Hazan, Sham~M Kakade, and Karan Singh.
\newblock Online control with adversarial disturbances.
\newblock \emph{arXiv preprint arXiv:1902.08721}, 2019{\natexlab{a}}.

\bibitem[Agarwal et~al.(2019{\natexlab{b}})Agarwal, Hazan, and
  Singh]{agarwal2019logarithmic}
Naman Agarwal, Elad Hazan, and Karan Singh.
\newblock Logarithmic regret for online control.
\newblock In \emph{Advances in Neural Information Processing Systems}, pages
  10175--10184, 2019{\natexlab{b}}.

\bibitem[Borodin and El-Yaniv(2005)]{borodin2005online}
Allan Borodin and Ran El-Yaniv.
\newblock \emph{Online computation and competitive analysis}.
\newblock cambridge university press, 2005.

\bibitem[Cohen et~al.(2018)Cohen, Hassidim, Koren, Lazic, Mansour, and
  Talwar]{cohen2018online}
Alon Cohen, Avinatan Hassidim, Tomer Koren, Nevena Lazic, Yishay Mansour, and
  Kunal Talwar.
\newblock Online linear quadratic control.
\newblock \emph{arXiv preprint arXiv:1806.07104}, 2018.

\bibitem[Goel and Wierman(2019)]{goel2019online}
Gautam Goel and Adam Wierman.
\newblock An online algorithm for smoothed regression and lqr control.
\newblock \emph{Proceedings of Machine Learning Research}, 89:\penalty0
  2504--2513, 2019.

\bibitem[Goel et~al.(2017)Goel, Chen, and Wierman]{goel2017thinking}
Gautam Goel, Niangjun Chen, and Adam Wierman.
\newblock Thinking fast and slow: Optimization decomposition across timescales.
\newblock In \emph{2017 IEEE 56th Annual Conference on Decision and Control
  (CDC)}, pages 1291--1298. IEEE, 2017.

\bibitem[Goel et~al.(2019)Goel, Lin, Sun, and Wierman]{goel2019beyond}
Gautam Goel, Yiheng Lin, Haoyuan Sun, and Adam Wierman.
\newblock Beyond online balanced descent: An optimal algorithm for smoothed
  online optimization.
\newblock In \emph{Advances in Neural Information Processing Systems}, pages
  1873--1883, 2019.

\bibitem[Hazan et~al.(2016)]{hazan2016introduction}
Elad Hazan et~al.
\newblock Introduction to online convex optimization.
\newblock \emph{Foundations and Trends{\textregistered} in Optimization},
  2\penalty0 (3-4):\penalty0 157--325, 2016.

\bibitem[Kailath et~al.(2000)Kailath, Sayed, and Hassibi]{kailath2000linear}
Thomas Kailath, Ali~H Sayed, and Babak Hassibi.
\newblock \emph{Linear estimation}.
\newblock Number BOOK. Prentice Hall, 2000.

\bibitem[Kalman et~al.(1960)]{kalman1960contributions}
Rudolf~Emil Kalman et~al.
\newblock Contributions to the theory of optimal control.
\newblock \emph{Bol. soc. mat. mexicana}, 5\penalty0 (2):\penalty0 102--119,
  1960.

\bibitem[Li et~al.(2018)Li, Qu, and Li]{li2018online}
Yingying Li, Guannan Qu, and Na~Li.
\newblock Online optimization with predictions and switching costs: Fast
  algorithms and the fundamental limit.
\newblock \emph{arXiv preprint arXiv:1801.07780}, 2018.

\bibitem[Li et~al.(2019)Li, Chen, and Li]{li2019online}
Yingying Li, Xin Chen, and Na~Li.
\newblock Online optimal control with linear dynamics and predictions:
  Algorithms and regret analysis.
\newblock In \emph{Advances in Neural Information Processing Systems}, pages
  14858--14870, 2019.

\bibitem[McDiarmid(1989)]{mcdiarmid1989method}
Colin McDiarmid.
\newblock On the method of bounded differences.
\newblock \emph{Surveys in combinatorics}, 141\penalty0 (1):\penalty0 148--188,
  1989.

\bibitem[Rauch et~al.(1965)Rauch, Tung, and Striebel]{rauch1965maximum}
Herbert~E Rauch, F~Tung, and Charlotte~T Striebel.
\newblock Maximum likelihood estimates of linear dynamic systems.
\newblock \emph{AIAA journal}, 3\penalty0 (8):\penalty0 1445--1450, 1965.

\bibitem[Stengel(1994)]{stengel1994optimal}
Robert~F Stengel.
\newblock \emph{Optimal control and estimation}.
\newblock Courier Corporation, 1994.

\end{thebibliography}



\end{document}